\documentclass[12pt,reqno]{amsart}
\usepackage{graphicx}
\usepackage{psfrag}
\usepackage{mathrsfs}
\usepackage{color}
\usepackage{amsmath,amsthm,amsfonts,amssymb,amscd, euscript}
\usepackage[all,cmtip]{xy}

\numberwithin{equation}{section}

\newcommand{\dist}{\operatorname{dist}}

\newcommand{\diff}{\operatorname{Diff}}

\newcommand{\diam}{\operatorname{diam}}
\newcommand{\supp}{\operatorname{supp}}

\newcommand{\e} {\varepsilon}

\def \TT {{\mathbb T}}
\def \RR {{\mathbb R}}

\newcommand{\Mane}{Ma\~n\'e}
\newcommand{\Diaz}{D\'\i az}
\newcommand{\eg}{e.g., }

\def \cN {{\mathcal N}}

\def \UL {\mathcal U_L}

\newtheorem{theorem}{Theorem}[section]
\newtheorem{remark}[theorem]{Remark}
\newtheorem{question}[theorem]{Question}
\newtheorem{conjecture}[theorem]{Conjecture}
\newtheorem{corollary}[theorem]{Corollary}
\newtheorem{lemma}[theorem]{Lemma}
\newtheorem{proposition}[theorem]{Proposition}
\newtheorem{definition}[theorem]{Definition}
\newtheorem{example}[theorem]{Example}
\def\ie{ i.e., }

\begin{document}

\thanks{ }

%\author{A. Tahzibi}
%\address{Departamento de Matem\'atica,
%  ICMC-USP S\~{a}o Carlos, Caixa Postal 668, 13560-970 S\~{a}o
%  Carlos-SP, Brazil.}
%\email{tahzibi@icmc.sc.usp.br}\urladdr{http://www.icmc.sc.usp.br/$\sim$tahzibi}

\thanks{}



\renewcommand{\subjclassname}{\textup{2000} Mathematics Subject Classification}

%\date{\today}

\setcounter{tocdepth}{2}

\title[Center Lyapunov exponents]{Center Lyapunov exponents in partially hyperbolic dynamics}

\author{Andrey Gogolev and Ali Tahzibi}
\thanks{The first named author was partially supported by NSF grant DMS-120494. The first named author also would
like to acknowledge the support provided by DeanÕs Research Semester Award. The second named author had the fellowship of CNPq and partially supported by FAPESP}
%\begin{abstract}
%The results can be extended to greater dimensions under some natural additional hypothesis.

%\end{abstract}

\maketitle

\tableofcontents

\section{Introduction}

Let $M$ be a compact smooth Riemannian manifold. Denote by
$\diff^r(M)$, $r\ge1$, the space of $C^r$ diffeomorphisms of $M$
endowed with $C^r$ topology and by $\diff^r_m(M)$ the subspace that
consists of diffeomorphisms which preserve volume $m$.

\begin{conjecture}[Pesin~\cite{pesin, pes07}] Consider a diffeomorphism $f\in\diff^r_m(M)$, $r>1$. Then arbitrarily close to $f$ in
$\diff^r_m(M)$ there exists a diffeomorphism $g$, which has nonzero Lyapunov exponents on a set of positive volume. Moreover, there is an open set $U$ containing $g$ and a residual subset of $U$ such that any $h$ in this subset has non-zero Lyapunov exponents on a positive measure set.
\end{conjecture}

This conjecture was motivated by the success of Pesin theory which
studies diffeomorphisms with non-zero exponents and, in particular,
establishes existence of an ergodic component of positive volume.

Pesin theory associates a rather subtle structure to a
diffeomorphism with non-zero exponents: Pesin sets, local stable and
unstable manifolds which vary only measurably and whose size
deteriorates along the orbits. Even though it is a widespread belief
that non-uniformly hyperbolic diffeomorphisms that carry this
structure are abundant in $\diff^r_m(M)$, concrete examples of such
diffeomorphisms are not very easy to come by. The subject of this
survey is partially hyperbolic
diffeomorphisms with non-zero center exponents. These systems
 provide a 
rather hands-on setting where some of the features of genuinely
non-uniformly hyperbolic behavior are present.

The $C^{r}, r>1$, topology in the second part of the above conjecture is crucial. By
\Mane-Bochi's result we have that $C^1-$generic volume preserving
diffeomorphisms of any surface is either Anosov or the Lyapunov
exponents of almost all points vanish. As the $2-$torus is the only
surface admitting Anosov diffeomorphism we conclude that for a
$C^1-$generic volume preserving diffeomorphism of any surface
different from the $2-$torus, the Lyapunov exponents of almost all points
vanish. From ergodic point of view $C^1$ dynamics is very different
since stable and unstable foliations could be non-absolutely continuous, which results in the failure of Pesin theory. We
remark that it is not even known whether or not $C^1$ Anosov
diffeomorphisms are ergodic and it is likely that they are not.

For large $r$ ($r>2\dim M+1$) one
cannot expect to find a dense set of diffeomorphisms whose exponents
are non-zero on a set of full volume since there are diffeomorphisms
with persistent regions that are occupied by codimension 1 invariant
tori (see \cite{CS, He, X, Y}). Still, for small $r$, e.g., $r=2$, one can hope for a dense
set of non-uniformly hyperbolic diffeomorphisms.

This survey is by no means a comprehensive one. Our goal is to give an overview of the field and explain several  major ideas in simplified setups. Many important topics were omitted or only touched upon briefly. We assume that the reader is familiar with the basic definitions and results from partially hyperbolic dynamics. This can be found in many sources, \eg \cite{PesHass, Pes2, RHRHU}.

%Let $f$ be a volume preserving (ergodic) partially hyperbolic diffeomorphism with one dimensional center bundle.
%If the center Lyapunov exponent is non zero, then for any $g$ close enough to $f$ there is a large invariant
%subset with non zero center Lyapunov exponent. However there are examples where such large invariant subset
%does not have full measure. Of course such examples can not be stably ergodic.
{\bfseries Acknowledgements.} We thank Amie Wilkinson who suggested in the first place that we write this survey. We would like to thank Anton Gorodetski and Victor Kleptsyn for useful communications. We also thank the referee for useful remarks which helped to improve our exposition.

\section{Abundance of non-zero Lyapunov exponents}
Consider a partially hyperbolic diffeomorphism $f$ of a 3-manifold $M$
with a fixed Finsler metric. By the Oseledets Theorem, the Lyapunov exponents
$$
\lambda^\sigma(f)(x)=\lim_{n\to\infty}\frac1n\log
J^\sigma(f^n)(x),\;\; \sigma=s, c, u,
$$
are well defined a.e. with respect to an invariant measure $\mu$. If $\mu$ is ergodic then by the Birkhoff
theorem
\begin{equation}
 \label{eq_exponent_formula}
\lambda^\sigma(f)(x)=\lim_{n\to\infty}\frac1n\sum_{i=0}^{n-1}\log J^\sigma(f)(f^ix)=\int_M
\log J^\sigma(f) d\mu,\;\; \sigma=s, c, u,
\end{equation}
for a.e. $x$. Using the last expression one can easily check that
the Lyapunov exponents are independent of the choice of the Finsler
metric. Also, this formula implies that $\lambda^\sigma$, $\sigma=s, c, u$, depend
continuously on the diffeomorphism from $\diff_\mu^r(M)$ in $C^1$ topology.

%More generally let $f \colon M \rightarrow M$ be a diffeomorphism of an $n$-dimensional manifold with dominated splitting $TM = E_1 \oplus \, E_2 \cdots \oplus, E_k, \dim (E_\sigma) = d_{\sigma}, 1 \leq \sigma \leq k$. Again by Oseledets Theorem for a.e. point 

Now consider a partially hyperbolic diffeomorphism $f \colon M \rightarrow M$ of an $n$-dimensional manifold 
that preserves the splitting $TM = E^s \oplus \, E^c \oplus \, E^u$ with $\dim (E^{\sigma}) = d_{\sigma}, \sigma =s, c, u$.  Again by the Oseledets Theorem for a.e. point $x \in M$ the Lyapunov exponents $\lambda^{\sigma}_i, 1 \leq i \leq d_{\sigma}, \sigma = s, c, u$ exist and 
$$
\sum_{i=1}^{d_{\sigma}} \lambda_i^\sigma(f)(x)= \lim_{n\to\infty}\frac1n\log
J^\sigma(f^n)(x) = \lim_{n\to\infty}\frac1n\sum_{i=0}^{n-1}\log J^\sigma(f)(f^ix),
$$

A weaker form of  hyperbolicity is  the existence of a dominated splitting.  A $Df$-invariant splitting $TM = E_1 \oplus E_2 \cdots \oplus E_k$ is called {\it dominated} if there exists $l \in \mathbb{N}$ such that $$ \frac{\| Df^l (x) v\|}{\|Df^l(x)w\|} \leq 1/2 $$ for any $x \in M$ and any $v \in E_i, w \in E_{i+1}, 1\leq i < k.$

It is easy to conclude from the definition that the largest Lyapunov exponent corresponding to $E_i$ is strictly less than the smallest exponent in the $E_{i+1}$ direction. So any diffeomorphism with dominated splitting of $k$-subbundles admits at least $k$ distinct Lyapunov exponents for a.e. point. However, some or all Lyapunov exponents along a fixed subbundle of the dominated splitting may coincide.

Next we will be concerned with the important case when $\mu$ is a
smooth measure. We will discuss techniques for
removing zero center exponent by a small perturbation in
$\diff_\mu^r(M)$.

\subsection{Removing zero exponent for smooth measures by a global perturbation}
\label{section_basic_construction} 
 We describe the construction of Shub and Wilkinson~\cite{SW} and incorporate
 some of the simplifications introduced in~\cite[p. 139]{BDV}.  Start with an algebraic partially hyperbolic automorphism of $\TT^3$. It turns out that for certain very explicit perturbations the computations are workable and the center exponent can be estimated using~(\ref{eq_exponent_formula}). 

 Let
$A\colon\TT^2\to\TT^2$ be an Anosov automorphism with positive
eigenvalues $\lambda<\lambda^{-1}$ and let $L_0$ be a partially
hyperbolic automorphism of $\TT^3$ given by
\begin{equation}\label{L0}
L_0(x_1,x_2,y)=(A(x_1,x_2),y).
\end{equation}
Our goal is to construct a small  volume preserving perturbation $f$
of $L_0$ that has non-zero center exponent.

Denote by $e_s$, $e_c$ and $e_u$ the unit constant vector fields on
$\TT^3$ that correspond to the eigenvalues $\lambda$, $1$ and
$\lambda^{-1}$ respectively.

Consider a stably ergodic partially hyperbolic skew product
$L_\varphi\colon\TT^3\to\TT^3$ of the form
\begin{equation}\label{Lphi}
L_\varphi(x_1,x_2,y)=(A(x_1,x_2),y+\varphi(x_1,x_2)),
\end{equation}
where $\varphi\colon\TT^2\to\TT^1$ is a smooth function $C^r$ close
to 0.
\begin{remark}  Burns and Wilkinson~\cite{BW} showed that stably ergodic skew products are open and dense in the space of skew products.
\end{remark}

Also consider a fibration $\pi\colon\TT^3\to\TT^1$ given by
$$
\pi(x_1,x_2,y)=y-x_1.
$$
(If we think of the tori $y=const$ as ``horizontal tori" then the
fibers of $\pi$ are ``diagonal tori" inside $\TT^3$.)

There exist unique numbers $a$ and $b$ such that vector fields
$e_s+ae_c$ and $e_u+be_c$ are tangent to the fibers of $\pi$. Let
$\hat\psi\colon\TT^1\to\RR$ be a non-constant smooth function $C^r$
close to $0$. Lift $\hat\psi$ to $\psi\colon\TT^3\to\RR$ along
the fibers: $\psi=\hat\psi\circ\pi$. Finally define a diffeomorphism
$h\colon\TT^3\to\TT^3$ by the formula
$$
h(x_1,x_2,y)=(x_1,x_2,y)+\psi(x_1,x_2,y)(e_s+ae_c).
$$
Diffeomorphism $h$ is a small translation in the direction $e_s+ae_c$ which is
constant on the fibers of $\pi$.

Let $f=h\circ L_\varphi$. Diffeomorphism $f$ is partially hyperbolic because it is a
small perturbation of $L_\varphi$. The derivative of $f$ can be
conveniently computed in the coordinates $(e_s+ae_c, e_c,
e_u+be_c):$
$$
DL_\varphi(x_1,x_2,y)=
\left(
  \begin{array}{ccc}
    \lambda & 0 & 0 \\
    \alpha+\varphi_s(x_1,x_2) & 1 & \beta+\varphi_u(x_1,x_2) \\
    0 & 0 & \lambda^{-1} \\
  \end{array}
\right)
$$
where $\alpha=a(1-\lambda)$, $\beta=b(1-\lambda^{-1})$, $\varphi_s$ and
$\varphi_u$ are derivatives of $\varphi$ in the direction  $e_s$ and
$e_u$ respectively;
$$
Dh=
\left(
  \begin{array}{ccc}
    1 & \psi_c & 0 \\
    0 & 1 & 0 \\
    0 & 0 & 1 \\
  \end{array}
\right)
$$
where $\psi_c$ is the derivative of $\psi$ in the direction $e_c$; then
$$
Df=Dh\circ DL_\varphi=
\left(
  \begin{array}{ccc}
    \lambda+(\alpha+\varphi_s)\psi_c\circ L_\varphi & \psi_c\circ L_\varphi & (\beta+\varphi_s)\psi_c\circ L_\varphi \\
    \alpha+\varphi_s & 1 & \beta+\varphi_u \\
    0 & 0 & \lambda^{-1} \\
  \end{array}
\right)
$$

We can see that $f$ is volume preserving. Notice also that $f$
preserves center-stable distribution $E^{cs}$ of $L_0$ spanned by
$e_s$ and $e_c$. Since center-stable distribution of $f$ is the
unique continuous 2-dimensional $f$-invariant distribution $C^0$
close to $E^{cs}$ we have that $E^{cs}$ is indeed the center-stable
distribution of $f$. Therefore, the center distribution of $f$ is
spanned by a vector field of the form $(\e, 1, 0)^t$ (written in the
coordinates $(e_s+ae_c, e_c, e_u+be_c)$). Function
$\e\colon\TT^3\to\RR$ is continuous and takes values in a small
neighborhood of $0$.

We equip $T\TT^3$ with Finsler metric given by the sup-norm in the
coordinates $(e_s+ae_c, e_c, e_u+be_c)$ so that
$Df(\e,1,0)^t=J^c(f)(\e\circ f,1,0)^t$. By a direct
computation
$$
Df\left(
     \begin{array}{c}
       \e \\
       1 \\
       0 \\
     \end{array}
   \right)
   =
   \left(
     \begin{array}{c}
       \e\lambda+\e(\alpha+\varphi_s)\psi_c\circ L_\varphi+\psi_c\circ L_\varphi \\
       \e(\alpha+\varphi_s)+1 \\
       0 \\
     \end{array}
   \right)
$$
Thus we obtain that $J^c(f)=1+\e(\alpha+\varphi_s)$ and
$$
\e\circ f=\psi_c\circ L_\varphi +\frac{\e\lambda}{J^c(f)}.
$$
From the last equation we see that $\e$ is a non-zero function.

Because $\alpha+\varphi_s>0$, we have that $J^c(f)>1$ when $\e>0$, and
$J^c(f)<1$ when $\e<0$. It follows that
$\e\lambda/J^c(f)\le\e\lambda$ and the equality holds if and only if
$\e=0$. Therefore $\e\circ f\le \psi_c\circ L_\varphi+\e\lambda$ and
after integrating we have a strict inequality:
$$
\int_{\TT^3}\e\circ f dm<\int_{\TT^3}\psi_c\circ L_\varphi
dm+\lambda\int_{\TT^3}\e dm=\lambda\int_{\TT^3}\e dm.
$$
Because $\lambda\in(0,1)$ we conclude that $\int_{\TT^3}\e dm<0$.

Now we can estimate the center Lyapunov exponent
\begin{multline}\label{eq_positive_exponent}
\lambda^c(f)=\int_{\TT^3}\log J^c(f) dm=\int_{\TT^3}\log
(1+\e(\alpha+\varphi_s)) dm\\<\int_{\TT^3}\e(\alpha+\varphi_s)
dm=\alpha\int_{\TT^3}\e dm<0.
\end{multline}
 
\subsection{Removing zero exponent for smooth measures by localized $C^1$-perturbations for general partially hyperbolic diffeomorphisms} 
\label{subs_BB}
Here we will work with general partially hyperbolic diffeomorphisms and expose a result of Baraviera and Bonatti~\cite{BB} that allows to perturb the center exponent
by a $C^1$-local perturbation. Their method allows to move any (center) exponent, not necessarily a zero exponent. Here the context is wider than that of the previous subsection 
but the perturbations are in $C^1$-topology. Similar construction in a more restricted setting was carried out by Dolgopyat~\cite[Appendix E]{D04}.

\begin{theorem}[\cite{BB}]
\label{th_BB}
 Let $(M, m)$ be a compact manifold endowed with a $C^r$ volume $m$, $r \geq 2$. Let $f$ be a $C^1$ $m$-preserving diffeomorphism that admits a dominated splitting $TM = E_1 \oplus \cdots \oplus E_k, k > 1$. Then there are $m$-preserving  $C^r$-diffeomorphisms $g,$ arbitrary $C^1$-close to $f$ for which 
 $$
 \displaystyle{\int_M \log J^i(g) (x) dm \neq 0}
 $$
 where $J^i(g)(x) := | \det Dg(x)| _{E_i} |.$
\end{theorem}

%Observe that $k > 1$ in the above theorem (the existence of a global dominated spitting) is a necessary condition 
%since the average of logarithm of full jacobian is zero. Also let us remark that by Bochi's result~\cite{bochi}, for $C^1$ generic surface 
%diffeomorphisms positive exponents arise only via dominated splitting. %That is, Bochi~\cite{bochi} had shown that 
%$C^1$-generic volume preserving surface diffeomorphism has zero Lyapunov exponents almost everywhere 
%unless it is Anosov.  As $2$-torus is the only surface admitting Anosov diffeomorphism,  $C^1$-generic volume 
%preserving diffeomorphism of $S \neq \mathbb{T}^2$ has two zero Lyapunov exponents almost everywhere.

%Observe that $k > 1$ in the above theorem (the existence of a global dominated spitting) is a necessary condition. By Bochi's %result for $C^1$-generic volume preserving diffeomorphisms on surfaces and aside from Anosov ones,  almost all point has two %zero Lyapunov exponents. As $2$-torus is the unique surface admitting Anosov diffeomorphism,  $C^1$-generic volume preserving %diffeomorphisms of $S$ $(S \neq \mathbb{T}^2)$has two zero Lyapunov exponents almost everywhere.

Recall that $\displaystyle{\int_M \log J^i(g) (x) dm}$  is equal to the integral of the sum of all Lyapunov exponents corresponding to the subbundle $E_i.$  For partially hyperbolic diffeomorphisms with one dimensional center bundle  the above theorem implies:
 \begin{corollary}
 There exists an open and dense subset $\mathcal{N}$ in the space of volume preserving partially hyperbolic diffeomorphisms with one dimensional center such that for any $f \in \cN$ we have 
 $$\displaystyle{\int_M \lambda^{c}(f)(x) dm \neq 0.}$$
  \end{corollary}

Let us mention that  C. Liang, W. Sun and J. Yang \cite{LSY} showed that if $f$ is partially hyperbolic with 
$TM = E^s \oplus E^c \oplus E^u$ with $\dim (E^c) = d_c$ then there is a volume preserving diffeomorphism $g$ $C^1$-close to $f$ 
such that 

$$ \int_M \lambda^c_i(g) (x) dm \neq 0 \quad \text{for all} \quad 1 \leq i \leq d_c,$$
where $\lambda^c_i(g) (x), 1 \leq i \leq d_c$, are the Lyapunov exponents corresponding 
to the center bundle. The idea 
is to perturb $f$ to a diffeomorphism $f_1$ so that the center bundle has  finest dominated splitting $E^{c_1} \oplus \cdots \oplus E^{c_k}$ in robust fashion. This means that every diffeomorphism $g$ close to $f_1$ has the finest dominated splitting of $E^c$ into the same number of subbundles with dimensions $d_i, 1 \leq i \leq k.$ 
By Baraviera-Bonatti result we can also suppose that $\displaystyle{\int_M \log J^i(g) (x) dm} \neq 0$ 
for  any such $g$. Since $E^{c_i}$ are elements of the finest dominated splitting, using a Bochi-Viana's argument, it is possible to perform a $C^1$-perturbation of $f_1$ to obtain the desired property for Lyapunov exponents.

 Now we will present a detailed proof of the Baraviera-Bonatti's result in the setting where $f$ is a linear partially hyperbolic automorphism of $\mathbb{T}^3.$ Then we will explain what adjustments are needed to obtain the result in the full generality.
 
 Let $f \colon \mathbb{T}^3 \rightarrow \mathbb{T}^3$ be a linear partially hyperbolic diffeomorphism with eigenvalues $0 < \lambda_s < \lambda_c < \lambda_u$. Let  $E^s \oplus \, E^c \oplus \, E^u$ be the corresponding splitting. We also consider a coordinate system corresponding to the this splitting with notation $(x, y, z)$. 
  
 Our goal is to perturb $f$ in a small ball  $B_r$ of radius $r$ around a non-periodic point $p_0$ and obtain a new volume preserving diffeomorphism $f_r$ such that
  $${\displaystyle \int_{\mathbb{T}^3} \lambda^{c} ({f_r})(x) dm >  \int_{\mathbb{T}^3} \lambda^{c} (f)(x) dm= \log(\lambda_c).}$$
  
  In fact we construct a one parameter family of volume preserving local perturbations $f_r$, $r\in [0,r_0]$,
 of $f$ such that the  average of unstable Lyapunov exponent of $f_r$ (which is equal to 
$\displaystyle{\int \log J^u_{f_r} dm}$) is strictly less than the unstable exponent of $f$ and the stable 
Lyapunov exponent of $f$ remains unchanged (the claim about stable exponent can be guaranteed only when $f$ is linear) after perturbation. Because $f_r$ is volume preserving  the sum of the 
Lyapunov exponents must be zero and consequently 
$$\displaystyle{ \int_{\mathbb{T}^3} \lambda^c ({f_r}) (x) dm > \int_{\mathbb{T}^3} \lambda^{c}(f)(x) dm.}$$

  As mentioned above after a local perturbation of linear partially hyperbolic 
automorphism we can decrease the unstable exponent. It is interesting to point out that it is not possible to increase it by a perturbation. In fact in the 
next section we show that the unstable Lyapunov exponent of any partially hyperbolic diffeomorphism on 
$\mathbb{T}^3$ is less than or equal to the unstable exponent of its linearization. 
  
%   $\int 
%  \lambda^u_{f_r} dm < \int \lambda^u_{f} dm, \int 
%  \lambda^s_{f_r} dm > \int \lambda^s_{f} dm$  and $
% \displaystyle{\int \lambda_f^u dm  - \int \lambda_{f_r}^u dm} > \displaystyle{\int \lambda_{f_r}^s dm  - \int \lambda_{f}^s dm} .
%$

\subsubsection{Construction of the perturbation}
Let $h \colon B(0, 1) \rightarrow B(0, 1)$ be a volume preserving diffeomorphism coinciding with the identity map on a neighborhood of the boundary of the unit ball   and $B_r , 0 <  r < 1$ an small ball around a non-periodic point $p_0$. We denote by  $h_r = \phi_r^{-1} \circ h \circ \phi_r$ where $\phi_r : B(0, r) \rightarrow B(0, 1)$ is a homothety of ratio $1/r.$ By definition $h_r$ is supported on $B_r$. We view $B_r$ as ball of radius $r$ about $p_0$ and extend $h_r$ by identity to the rest of $\mathbb T^3$. Let $f_r = f \circ h_r$. Then $f_r$ is a small $C^1$ volume preserving perturbation of $f.$ Observe that $\|h_r - Id\|_{C^1} = \|h - Id\|_{C^1}$ but $\|h_r - Id\|_{C^k}, k > 1$ depends on $r$ and consequently this method does not provide $C^k$-perturbations to remove zero Lyapunov exponent. The following question  remains open:

\begin{question}
Is it possible to make a local $C^r$-perturbation $(r > 1)$ to remove zero Lyapunov exponent of the volume?
\end{question}

We require that $h$ preserves $E^s$ coordinate, that is, $ \pi_s h(x,y,z) = x.$ 
Let us prove a local estimate which sheds some light on the behavior of center Lyapunov exponent of $f_r$. 
Let $Dh(p) (e_u) = h^u(p) e_u + h^{c}(p) e_c$ for any $p \in B(0, 1)$ where $e_u$ and $e_c$ stand for unit 
vectors in $E^u$ and $E^c$. Also define $h_r^u(q) := h^u(\phi_r(q))$.

\begin{lemma}
Let $B(0, 1)$ be the unit ball of $\mathbb{R}^3$. Consider any volume preserving diffeomorphism $h$ of $B(0, 1)$ coinciding with the identity map on a neighborhood of the boundary of the ball and preserving $E^s$ coordinate. Assume that $0 < \| h - Id\|_{C^1} < 1$ then
$$
 \int_{B(0, 1)} \log   h^u(p) dm < 0.
$$
\end{lemma}

To prove the above lemma we take any segment $\gamma$ tangent to $E^u$ direction  and joining two  
boundary points of $B(0, 1)$.  Let $\gamma : J \subset \mathbb{R} \rightarrow B(0, 1)$ be a parametrization 
by the arc length, that is $\|\gamma'(\xi)\|=1$.  Denote by $\pi_u$ the orthogonal projection on $\gamma(J)$. Since $h$ coincides with identity close to the boundary we have that 
the curve $\pi_u (h (\gamma))$  has the same image as $\gamma$, but, possibly, different parametrization. These curves have the same length, \ie
$$
 l(\gamma)=\int_{J}  \|(\pi_u \circ h \circ \gamma)^{'}(\xi)\| d\xi = \int_J  h^u(\gamma(\xi)) d\xi.
$$
By Jensen inequality for the probability measure $\frac{d\xi}{l(\gamma)}$, 
$$
 \int_{J} \log  h^u(\gamma(\xi)) \frac{d\xi}{l(\gamma)} \leq \log \int_{J} h^u(\gamma(\xi)) \frac{d \xi}{l(\gamma)} = 0.
$$

We claim that the above inequality is strict. Indeed, in the case of equality $ h^u$ is constant equal to $1$. Because $h$ coincides with identity close to the boundary we conclude that $h$ preserves $E^u-$coordinate. As $h$ is volume preserving and preserves $E^s$ and $E^u$ coordinates, the derivative of $h$ in $E^c$-direction is also equal to $1$. Again, because $h$ is identity close to the boundary, it preserves $E^c$-coordinate too. This implies that $h$ is identity which is a contradiction.
\subsubsection{An auxiliary cocycle}

%Now we calculate the average of unstable Lyapunov exponent of $f_r.$ Clearly, the  center and unstable bundles of $f_r$ 
%may be different from those of $f$. However, as $f_r$ is $C^1-$close to $f$ then the center and unstable bundles of $f_r$ in each point $p \in \mathbb{T}^3$ are respectively near to $E^s$ and $E^u.$
 %We introduce a cocycle over $f_r$ which lets invariant  $E^u$. Moreover the Lyapunov exponent of the cocycle on almost every point coincide with the unstable Lyapunov exponent of $f_r.$   
 
%For any $p, Df(e_u) = \lambda^u e_u$ and cocycle $F_r(p, e_u) = (f_r(p), \lambda_r(p) e_u)$ where $\lambda_r(p)$ is defined as follows. The cocycle will be defined in such a way that 

Now we calculate the average of the unstable Lyapunov exponent of $f_r$. Note that the unstable and center 
distributions are $C^1$-close to, but different from those for $f$. This makes it difficult to estimate
 $$
\int_{\mathbb T^3} \log J^u (f_r) (p) dm
$$
directly. The trick is to substitute $J^u_{f_r}$ by a function $\lambda_r$ which is easier to analyze, but at the 
same time
$$
\int_{\mathbb T^3} \log \lambda_r(p) dm = \int_{\mathbb T^3} \log J^u (f_r) (p) dm.
$$

We will think of $\lambda_r$ as the generator of a cocycle $F_r$ over $f_r$ acting on $E^u$,
$$
F_r\colon (p, e_u)\mapsto (f_r(p),\lambda_r(p)e_u).
$$
This way $F_r(p, e_u)=\lambda_r(p) e_u$.  By
 definition, the Lyapunov exponent of the cocycle at $p$ is 
$$
 \lim_{n \rightarrow \infty} \frac{1}{n} \log \| F_r^n (e_u) \|.
$$ 
Now let us define the cocycle:
\begin{itemize}
\item The action of $F_r$ coincides with $Df$ outside $B_r \cup f^{-1}(B_r)$ that is, $\lambda_r(p) := \lambda_u$.
 \item In $B_r$, the action of $F_r$ is the action of $Df$ on the projection of  $Dh_r(e_u)$ on $e_u$ parallel to the center bundle of $f$,  that is, $\lambda_r(p) := \lambda_u h_r^u(p)$. 
\item For any $p \in f^{-1}(B_r)$ if the negative orbit of $p$ is disjoint from $B_r$ then set $\lambda_r(p) = \lambda_u.$ Otherwise let $n(p)$ be the smallest integer for which there exist $q\in B_r$  such that $p = f^{n(p)}(q)$. Let $\tilde q = h^{-1}({q})$.
Recall that $$Dh_r (\tilde{q})(e_u) = h^u_r(\tilde{q}) e_u + h_r^c(\tilde{q}) e_c.$$ On one hand we  have
$$F_r^{n(p)}(e_u) =  h^u_r(\tilde{q}) \lambda_u^{n(p)}  e_u$$
and on the other hand, using chain rule for $f_r = f \circ h_r$, we have
$$
 Df^{n(p)}_r (e_u) = F_r^{n(p)}(e_u) + \lambda_c^{n(p)} h_r^c(\tilde{q}) e_c.
$$
Observe that the second summand above belongs to $E^c$ which may be different from the center direction of the perturbed diffeomorphism.  Let $w(p)$ be the projection of $\lambda_c^{n(p)} h_r^c(\tilde{q}) e_c$ on $e_u$ direction parallel to the new center direction at $p$ for $f_r.$  Observe that in the $n(p)$ iterates of the cocycle we ignored the vector in the $e_c$ direction which should be considered for the new diffeomorphism. So consider the correction term $$A(p) := \frac{\|F_r^{n(p)}(e_u) + w(p)\|}{\|F_r^{n(p)}(e_u)\| } = 1 + \frac{\|w(p)\|}{ h_r^c (\tilde{q}) \lambda_u^{n(p)} }$$  and define $\lambda_r(p):= A(p) \lambda_u.$ 
\end{itemize}

%\begin{equation}
%  \label{cocycle}
%  \lambda_r(p)=\left\{
%\begin{array}{lll}
%\lambda(p) \qquad &  p \notin B_r \cup f^{-1}(B_r)\\
%\lambda(h_r(p)) h_r^u(p) \qquad & p \in B_r \\
%\lambda(p) A(p) \qquad & p \in f^{-1}(B_r).
%
%\end{array} \right.
%\end{equation}

By the above definitions, for any point $p \in B_r$ and $N > 0$ such that $f^N(p) \in B_r$, we have that $F_r^{N} (e_u)$ is just the projection of $Df_r^N(e_u)$ on  $e_u$ direction along the new center stable bundle of $f_r$. As this bundle is transversal to $e_u$ and $Df_r^N(e_u)$ belongs to an unstable cone field this projection affects the norm in uniformly bounded way and the Lyapunov exponent of $f_r$ along $E^u_{f_r}$ coincides with the Lyapunov exponent of the cocycle $F_r.$

Let us state the key proposition which shows the effect of the perturbation on the  unstable Lyapunov exponent.
\begin{proposition}\label{prop_estimate}
For sufficiently small $r > 0$, $  \displaystyle{\int_{\mathbb{T}^3} \log J^u(f_r) dm} < \log(\lambda_u)$.
\end{proposition}

By the above proposition, after perturbation the average of the unstable Lyapunov exponent decreases. 
By construction the perturbation does not change the stable Lyapunov exponent. 
At the same time the volume is preserved and the sum of Lyapunov exponents remains zero. Hence,
after perturbation the center Lyapunov exponent  increases.  

To finish the proof of Theorem~\ref{th_BB} it remains to prove Proposition~\ref{prop_estimate}.
\begin{proof}
As  $  \displaystyle{\int_{\mathbb{T}^3} \log J^u (f_r) (x) dm = \int_{\mathbb{T}^3} \log \lambda_r(x) dm }$ and $  \lambda_r(p) = \lambda_u$ for $p \notin B_r \cup f^{-1}(B_r)$ we have
\begin{align*}
\log(\lambda_u) & - \int_{\mathbb{T}^3} \log J^u (f_r) dm \\ &= \int_{B_r} \log(\lambda_u) -\log(\lambda_r(p)) dm + \int_{f^{-1}B_r} \lambda_u -\log(\lambda_r(p)) dm \\
&= -vol(B_r) \int_{B(0, 1)} \log(h^u (p)) dm  - \int_{f^{-1}B_r} \log(A(p)) dm.
\end{align*}
where the last equality comes from the definition of $\lambda_r.$ The first summand is the product of $-vol(B_r)$ and a negative constant that does not depend on $r$. We estimate the second term from above
$$\left| \int_{f^{-1}B_r} \log(A(p)) dm\right|   \leq   vol(B_r) \max_{p \in f^{-1}B_r} \log(A(p)).$$

Note that up till now we did not use ``partial hyperbolicity" of $f$. 
Now we will use domination (here it is just $\lambda_c < \lambda_u$) 
to show that the maximum above is exponentially small. Let $n_r$ be the least strictly positive integer $n$ such that $f^{n_r} B_r \cap B_r \neq \varnothing.$ Obviously $n_r\to\infty$ as $r\to 0$.

\begin{lemma}
  There is exists $0 < \alpha < 1$ and a constant $C > 0$ such that for any $r > 0$
  $$
   \max_{ p \in f^{-1}B_r} |\log(A(p))| \leq C \alpha^{n_r}.
  $$
\end{lemma}
We show that $|A(p) -1| \leq C_0 \alpha^{n_r}.$ By definition of $A(p)$ we have to show
$$
  \frac{\|w(p)\|}{ h_r^u(\tilde{q}) \lambda_u^{n(p)} }  \leq C_0 \alpha^{n_r}.
$$
As the projection of $e_c$ on $e_u$ along the new center direction has a uniformly bounded norm for all points it is enough to show that 
$$
\frac{h_r^c (\tilde{q}) \lambda_c^{n(p)} }{ h^u_r (\tilde{q}) \lambda_u^{n(p)} }  \leq C_1 \alpha^{n_r}.
$$

Take $\alpha = \frac{\lambda_c}{\lambda_u}$. Then the inequality follows because $n(p) \leq n_r$ and the partial derivatives of $h$ are bounded away from below and above on $B(0, 1).$

%That is, after similar $C^1$-local perturbation  $$\displaystyle{\int_{\mathbb{T}^3} \log \|Df^u_r\| dm} < \displaystyle{\int_{\mathbb{T}^3} \log \|Df^u\| dm},$$ the stable exponent remains unchanged and consequently $$\displaystyle{ \int_{\mathbb{T}^3} \lambda^c ({f_r}) (x) dm(x) > \int_{\mathbb{T}^3} \lambda^{c} ({f}) (x) dm(x).}$$ 

  \end{proof}

% \begin{remark} {\rm In the above proof the fact that $f$ is linear was essentially used in order to fix a constant partially hyperbolic decomposition.  In fact if the invariant sub bundles were differentiable we could find suitable coordinates and work exactly like the linear case. In the general case, Since the stable foliation is not smooth it is not clear whether it is possible to make the stable exponent unchanged after perturbation.  However if we consider a very small ball of radius $r= r(p_0)$ around $p_0$ the partially hyperbolic decomposition inside the ball is near to constant. By a similar  local perturbation in this ball, the unstable Lyapunov exponents drops much more than the stable one (possibly) increases. It turns out again that the center Lyapunov exponent increases after perturbation. }
% \end{remark}
 
\begin{remark} Considering inverse of diffeomorphism $f$ and  similar perturbation as above  it would 
be also possible to make the average of center exponent decrease after a small perturbation.
\end{remark}

The hypothesis of $f$ being linear  may seem to be very far from general partial hyperbolicity. 
However, if $f$ is partially hyperbolic with non-periodic point $p_0$ then in a neighborhood of $p_0$ the 
splitting $E^s \oplus E^c \oplus E^u$   is close to constant splitting and the same conclusions about Lyapunov 
exponents of the perturbation can be made. One argues  as follows:

 The perturbation is again $f \circ h_r$ such that $h_r$ is the conjugated by homothety to $h : B(0, 1) \rightarrow B(0, 1)$ which is defined exactly as before. Although $h$ preserves the direction parallel to $E^s(p_0)$, the stable Lyapunov exponent of $f_r$ is not necessarily equal to $f.$ This is because the stable subbundle of $f$ is not constant. However, in a very small scale the partially hyperbolic decomposition is almost constant.  It can be proved that

$$
 \lim_{r \rightarrow 0} \frac{\displaystyle{\int_{\mathbb{T}^3} \log J^u (f_r) dm  - \int_{\mathbb{T}^3} \log J^u (f) dm}}{vol(B_r)}= \int_{B(0, 1)} \log h^u(p) dm < 0
$$
and 
$$
 \lim_{r \rightarrow 0} \frac{\displaystyle{\int_{\mathbb{T}^3} \log J^s (f_r) dm  - \int_{\mathbb{T}^3} \log J^s (f) dm}}{vol(B_r)}= 0.
$$
This yields that for any small enough $r > 0$
$$
 \displaystyle{\int_{\mathbb{T}^3} \log J^u (f) dm  - \int_{\mathbb{T}^3} \log J^u (f_r) dm} > \left|\displaystyle{\int_{\mathbb{T}^3} \log J^s (f_r) dm  - \int_{\mathbb{T}^3} \log J^s (f) dm}\right |.
$$

As the volume is preserved, the center exponent should increase after perturbation. 
\begin{question}
Take a symplectic partially hyperbolic diffeomorphism $f$ with two dimensional center bundle and zero center Lyapunov exponents, i.e $\lambda^c_1 (x, f) = \lambda^c_2(x, f)= 0$ a.e $x$. Is it possible to $C^1-$perturb $f$ to a symplectomorphism $g$ such that $\int_M \lambda^c_{1}(x, g) d m(x) \neq 0?$
\end{question}

%Indeed, let $h : B(0, 1) \rightarrow B(0, 1)$ be as before and $h_r$ be the conjugation of $h$ with homothety of ratio $r$ and $f_r = f \circ h_r.$ \\ We have shown that $\int_{B(0, 1)} \log h^u(p) dm(p) < 0.$  Let $h^u_r(p)$ be the modulus of the linear map from $E^u_f(p)$ to $E^u_f(h_r(p))$ obtained as the composition of $Dh(p)$ with the projection of $T_{h_r(p)}M$ on $E^u_f(h_r(p))$ parallel to $E^{cs}_f(h_r(p)).$ Analogously we define $h_r^s(p).$ Then we have
%$$
%\lim_{r \rightarrow 0 } \frac{\displaystyle{\int_{B_r} \log(h^u_r(p)) dm(p)}}{vol(B_r)} = \int_{B(0, 1)} \log h^u(p) dm(p) < 0,
%$$ and 
%$$
%\lim_{r \rightarrow 0 } \frac{\int_{B_r} \log(h^s_r(p)) dm(p)}{vol(B_r)} = 0.
%$$

\subsection{Removing zero exponents for SRB measures} \label{section_srb} The original approach of Shub and Wilkinson in~\cite{SW} was to consider a special one parameter family of volume preserving partially hyperbolic diffeomorphism $f_\e\colon\TT^3\to\TT^3$ through the partially hyperbolic automorphism $f_0$ given by~(\ref{L0}) and then to obtain a second order expansion for the center Lyapunov exponent for the volume
$$
\lambda^c(f_\e)=K\e^2+o(\e^2),
$$
with an explicit non-zero $K$. This approach was further pursued by Ruelle~\cite{R}. Ruelle considered general one parameter families $f_\e\colon\TT^n\to\TT^n$ through a partially hyperbolic automorphism with one dimensional center $f_0\colon\TT^n\to\TT^n$. He obtained explicit second order expansions for the center Lyapunov exponent for volume preserving families as well as dissipative families $f_\e$ assuming that they have unique SRB measures $\mu_\e$.

Dolgopyat~\cite{D04} studied the case where $f_0$ is the time one  a geodesic flow on a closed surface of negative  curvature. Based on earlier work on existence and uniqueness of SRB measures for partially hyperbolic diffeomorphisms~\cite{ABV, D00} he proved the following result.
\begin{theorem}[\cite{D04}]
Let $f_\e$ be a generic one parameter $C^1$ family of partially hyperbolic diffeomorphisms passing through $f_0$. Then $f_\e$ has a unique SRB measure $\mu_e$ for small $\e$. Moreover, one has the following second order expansion for the center Lyapunov exponent of $\mu_\e$
$$
\lambda^c(f_\e,\mu_\e)=K\e^2+o(\e^2), K\neq 0.
$$
\end{theorem}

\section{Lyapunov exponents and linearizations}
In the previous section we described perturbations that remove zero center Lyapunov exponent. Recall that the idea in both (local and global perturbations) methods was to ``borrow" exponents from the unstable bundle to the center bundle. However, as we remarked before, for all perturbations that we did, the Lyapunov exponent of unstable bundle decreased. 

So let us shift our attention for a moment to the unstable Lyapunov exponent. Consider the functional 
\begin{equation}
\label{eq_Lambda}
 \Lambda:  f \rightarrow  \int_M \lambda^u (f) (x) dm(x).
\end{equation}

  In what follows we prove a result which, in particular, implies that linear partially hyperbolic diffeomorphisms are local maximum point for $\Lambda.$ 
  
  \begin{theorem}[\cite{MT}] \label{mt} 
Let $f : \mathbb{T}^3 \rightarrow \mathbb{T}^3$ be a $C^2$ partially hyperbolic diffeomorphism and $A$ be its linearization. Then $\lambda^u(f) (x) \leq \lambda^u (A) $ and $\lambda^s(f) (x) \geq \lambda^s (A)$ for Lebesgue almost every $x \in \mathbb{T}^3.$ 
\end{theorem}
  
  We recall that, using geometric growth of foliations,  Saghin and Xia~\cite{SX} proved that the unstable Lyapunov exponent does not increase after perturbation of a linear partially hyperbolic diffeomorphism. 
We prove a global version of their result.  The claim for the stable Lyapunov exponent comes out just by taking $f^{-1}$ instead of $f.$ 

  Firstly, let us recall some useful facts on the partially hyperbolic diffeomorphisms of $\mathbb{T}^3.$ Every diffeomorphism of the torus $f : \mathbb{T}^3 \rightarrow \mathbb{T}^3$ induces an automorphism of the fundamental group and there exists a unique linear diffeomorphism $A$ which induces the same automorphism on $\pi_1(\mathbb{T}^3).$ The diffeomorphism $A$ is called the linearization of $f$. One can choose a lift of $f$ and $A$ to the universal cover $\mathbb{R}^3$ which we denote again by $f$ and $A$. The lifts are also partially hyperbolic and preserve invariant foliations $\mathcal{F}^{\sigma}, \sigma \in \{s, c, u\}.$   
  An important geometric property of invariant foliations $\mathcal{F}^{\sigma}$ (in the universal cover) is their {\it quasi-isometric property}: There exists a universal $Q > 0$ such that $\|x - y\| \geq Q d^{\sigma} (x, y)$ for any  $ x, y \in \mathbb{R}^ 3, y \in \mathcal{F}^{\sigma}(x)$, where $d^{\sigma}$ stands for the inherited Riemannian metric along the leaves of $\mathcal{F}^{\sigma}.$ See \cite{H} and \cite{BBI}.

Let us state two basic propositions which are used in the proof of the above theorem.

  \begin{proposition} [\cite{H}] \label{H2} Let $f : \mathbb{T}^3 \rightarrow \mathbb{T}^3$ be a partially hyperbolic diffeomorphism and let $A: \mathbb{T}^3 \rightarrow \mathbb{T}^3$ be the linearization of $f$ then for each $k \in  \mathbb{Z}$ and $C > 1$ there is an $M > 0$ such that for $x, y \in \mathbb{R}^3$,
$$||x -  y||> M \Rightarrow
\frac{1}{C} <\frac{|| f^k(x) -  f^k(y)||}
{||A^k(x) - A^k(y)||}
< C.$$
\end{proposition}

\begin{proposition} [\cite{MT}] \label{linalg}
 Let $f : \mathbb{T}^3 \rightarrow \mathbb{T}^3$ be a partially hyperbolic diffeomorphism and $A: \mathbb{T}^3 \rightarrow \mathbb{T}^3$ the linearization of $f.$ For all $n \in \mathbb{Z}$ and $\epsilon > 0$ there exists $M$ such that for $x, y \in \mathbb{R}^3$ with $y \in \mathcal{F}^{\sigma}_x$ and $||x -  y||> M$ then
 $$
   (1 - \varepsilon)e^{n\lambda^{\sigma}(A )} ||y -x|| \leq \|A^n(x) - A^n(y)\| \leq (1 + \varepsilon)e^{n\lambda^{\sigma} (A) } ||y -x||,
 $$
where $\lambda^{\sigma} (A)$ is the Lyapunov exponent of $A$ corresponding to $E^{\sigma}$ and $\sigma \in \{s, c, u\}.$
 \end{proposition}

\begin{proof}  
We prove the statement on $\lambda^u(f)$ of Theorem \ref{mt}. Suppose by contradiction that there is a  positive volume set $Z \subset \mathbb{R}^3,$ such that, for every $x \in Z$ we have  $\lambda^u(f)(x) > \lambda^u (A).$ We can take $Z$ such that, there exists $\epsilon > 0$ such that $\lambda^u (f) (x) > \lambda^u (A) + 2 \log (1+ 2\epsilon)$ for all $x \in Z.$
Since $f$ is  $C^2,$ the unstable foliation $\mathcal{F}^u$ for $f$ is  absolutely continuous. In particular there is a positive volume set $B \in \mathbb{R}^3$ such that for every point $p \in B$ we have

\begin{equation}
 m_{\mathcal{F}^u_p}(\mathcal{F}^u_p \cap Z) > 0,
\label{1}
\end{equation}
where $m_{\mathcal{F}^u_p}$ is the induced volume on the unstable leaf. Now consider a segment $[x,y]_u \subset \mathcal{F}^u_p $ satisfying
 $m_{\mathcal{F}^u_p}([x,y]_u \cap Z) > 0$ such that $d^u(x, y) \geq \frac{M}{Q}$, where $M$ is as required in propositions \ref{H2}, \ref{linalg} and $Q$ is the quasi isometric constant. So, we have $\|x -y\| \geq M$ and  by choosing $M$ large enough we have:
$$||Ax - Ay|| \leq (1 + \varepsilon)e^{\lambda^u (A) } ||y -x|| $$
and
$$\frac{|| fx - fy|| }{ ||Ax - Ay||} \leq 1 + \varepsilon.$$
  The above equations imply that $$ || fx - fy|| \leq (1+ \varepsilon)^2 e^{\lambda^u (A)} || y - x||.$$
Inductively, we assume that for  $n \geq 1$ we have
\begin{equation}
|| f^nx - f^ny|| \leq (1+\varepsilon)^{2n} e^{n \lambda^u (A)}|| y - x||. \label{induction}
\end{equation}
Since $f$ expands uniformly  on the $u-$direction we have $\|f^nx- f^ny\| > M$ and hence

\begin{eqnarray*}
|| f(f^nx) - f(f^ny)|| & \leq& (1+\varepsilon)|| A(f^nx) - A(f^ny)|| \\  &\leq& (1 + \varepsilon)^2 e^{\lambda^u(A)} || f^nx - f^n y||\\ &\leq&
 (1+\varepsilon)^{2(n+1)} e^{(n+1)\lambda^u (A)}.
\end{eqnarray*}

For each $n > 0,$ let $A_n \subset Z$ be the following set
$$A_n = \{ x \in Z \colon\;\; \|D^uf^k \| \geq (1+2\varepsilon)^{2k} e^{k\lambda^u (A)} \;\; \mbox{for any} \;\; k \geq n\}. $$
We have $m(Z) > 0$ and by the choice of $\epsilon$, $A_n \uparrow Z.$
Consider a large $n$ and  $\alpha_n > 0$ such that $ m_{\mathcal{F}^u_p} ([x,y]_u \cap A_n) = \alpha_n m_{\mathcal{F}^u_p} ([x,y]_u).$
 We have $\alpha_n \geq \alpha_0 > 0$ for for every large $n > 1.$ Then

\begin{eqnarray}
||f^nx - f^ny || &\geq&  Q \displaystyle\int_{[x,y]_u \cap A_n} ||D^uf^n(z)|| d m_{\mathcal{F}^u_p}(z) \geq  \\ &\geq&
 Q (1+ 2\varepsilon)^{2n}
e^{n \lambda^u(A) } m_{\mathcal{F}^u_p} ([x,y]_u \cap A_n) \\ &\geq& \alpha_0 Q (1 + 2\varepsilon)^{2n} e^{n\lambda^u (A)} \|x-y\|. \label{conclusion}
\end{eqnarray}
The inequalities $(\ref{induction})$ and $(\ref{conclusion})$ contradict each other. Therefore we obtain $\lambda^u(f)(x) \leq \lambda^u (A),$ for almost every $x \in \mathbb{T}^3.$
Considering the inverse $f^{-1} $ we conclude that
$ \lambda^s (A) \leq  \lambda^s(f)(x)$
for  almost every $x \in \mathbb{T}^3.$\end{proof}

\begin{conjecture}
Let $f$ be a $C^2$ volume preserving Anosov diffeomorphism of $\mathbb{T}^2$ and suppose that $f$ is a local maximum for the functional $\Lambda : \diff^2_m(\mathbb{T}^2) \rightarrow \mathbb{R}$ given by~\textup{(\ref{eq_Lambda})}, then $f$ is $C^1$ conjugate to a linear toral automorphism. 
\end{conjecture}

We would like to point out that a ``global version'' of this conjecture holds true. To see this fix a hyperbolic toral automorphism $A\colon\TT^2\to\TT^2$ and let $\EuScript X_A$ be the space of Anosov diffeomorphism that are homotopic to $A$. Then each $f\in \EuScript X_A$ is conjugate to $A$ by a homeomorphism $\alpha_f$. Notice that $\Lambda(f)$ is simply the metric entropy of $f$. Hence
$$
\Lambda(f)=h_m(f)=h_{\alpha_f^*m}(L)\le h_m(L)=h_{top}(L).
$$
The last inequality is the variational principle. Hence $\Lambda$ attains the global maximum on $L$. Moreover, if $h_{\alpha_f^*m}(L)= h_m(L)$ then the conjugacy $\alpha_f$ has to be volume preserving and, hence, by work of de la Llave, Marco and Moriy\'on~\cite{MM,dlL}, $\alpha_f$ is $C^{1+\varepsilon}$, $\varepsilon>0$. We conclude that $\Lambda$ attains a global maximum on $f\in\EuScript X_A$ if and only if $f$ is $C^1$ conjugate to $A$.

\section{Non-removable zero exponents}
The evidence presented earlier leads to the belief that generically partially hyperbolic diffeomorphisms have non-zero center exponents with respect to natural measures such as volume or an SRB measure.
However, if one considers all ergodic measures then it is natural to expect existence of a measure with some (or all) zero center exponents.

Consider a partially hyperbolic automorphism $L\colon\TT^3\to\TT^3$ 
\begin{equation}\label{L}
L(x_1,x_2,y)=(A(x_1,x_2), y)
\end{equation}
where $A$ is an Anosov automorphism. And let $\UL$ be a $C^1$ small neighborhood of $L$ (precise definition of $\UL$ appears later).
\begin{question}
\label{q_path_conn2}
 Given an ergodic volume preserving diffeomorphism $f\in\UL$, is it true that the space of $f$-invariant ergodic measures equipped with weak$^*$ topology is path connected?
\end{question}
\begin{remark}
 The space of ergodic measures of a transitive Anosov diffeomorphism is path connected~\cite{sigmund}.
\end{remark}
\begin{remark}
 Let $g\colon S^1\to S^1$ be a diffeomorphism with two hyperbolic fixed points, an attractor and a repeller. Then the 
space of ergodic measures of diffeomorphism ${A\times g\colon \TT^3\to\TT^3}$ has two connected components. Moreover, this property
is $C^1$-stable---the space of ergodic measures of any sufficiently small perturbation of $A\times g$ has at least
two connected components. This shows that it is essential that $f$ is conservative for the above question.
\end{remark}

The motivation for Question~\ref{q_path_conn2} comes from the following observation. Pick a diffeomorphism $f\in\UL$
such that the restriction of $f$ to an invariant center leaf has a hyperbolic attracting point $a$ and a hyperbolic repelling point $b$.
Assume that there is a path $\mu_t$, $t\in[0,1]$, of ergodic measures connecting the atom at $a$ and the atom at $b$.

The center Lyapunov exponent 
$$
\lambda^c(\mu_t)=\int_{\TT^3}\log J^c_f d\mu_t
$$
depends continuously on $t$ and, by the intermediate value theorem, there exists $t_0\in[0,1]$ such that the center exponent of $\mu_{t_0}$ is zero.

Note that the fixed points $a$ and $b$ have continuations in a small $C^1$ neighborhood $\mathcal V$ of $f$ and, hence, the 
same argument applies for diffeomorphisms in $\mathcal V$.

Thus, for example, a positive answer to Question~\ref{q_path_conn2} implies that there exists an open set of volume preserving
diffeomorphisms near $L$ each of which has an ergodic measure with zero center exponent.

\begin{remark}
 In fact, the above argument can be applied to any $f\in\UL$. Indeed, if $f$ has positive (negative) center Lyapunov 
exponents for all invariant measures then it must be uniformly hyperbolic~\cite{AAS}.
\end{remark}

Even though the answer to Question~\ref{q_path_conn2} is unknown one can still proceed with a similar idea, which is to consider a sequence (rather than a path) of measures whose center exponent tend to zero and derive results on existence of measures with zero center exponents.
%Certain weaker version of Question~\ref{q_path_conn2} has a positive answer which is sufficient to derive the results 
%on existense of measures with zero center exponents.

\subsection{The results}

We consider an automorphism $L\colon \TT^3\to\TT^3$ given by~\eqref{L} and a small $C^1$ neighborhood
$\UL$ of $L$ in $\diff^1(\TT^3)$ chosen so that for every $f\in\UL$ Hirsch-Pugh-Shub structural stability applies and yields
a homeomorphism $H\colon\TT^3\to\TT^3$ that conjugates $L$ and $f$ on the space of center leaves. From now on 
neighborhood $\UL$ will be fixed.
\begin{theorem}[\cite{IG1, IG2, IGKN, KN}]\label{th_IGKN}
 There exists a $C^1$ open set $\mathcal V\subset\UL$ such that for every $f\in\mathcal V$ there exists 
an ergodic measure $\mu_f$ of full support with zero center Lyapunov exponent.
\end{theorem}

The measure $\mu_f$ is constructed as weak$^*$ limit of atomic measures which are supported on certain periodic orbits
of $f$. The construction rests on so called Gorodetski-Ilyashenko strategy that aims at showing that
properties of random dynamical systems (higher rank free actions) could be observed in partially 
hyperbolic diffeomorphisms and, more generally, locally maximal partially hyperbolic invariant sets.

We will explain this strategy and outline the proof of Theorem~\ref{th_IGKN} in the next subsection. So far
the work on realizing Gorodetski-Ilyashenko strategy was devoted to higher rank free actions on the
circle $S^1$. This corresponds to partially hyperbolic attractors whose center foliation is a
foliation by circles. It would be very interesting to go beyond one dimensional case.

Recently Gorodetski and \Diaz~have announced that one can make sure that set $\mathcal V$ from
Theorem~\ref{th_IGKN} contains diffeomorphism $f$ constructed in Section~\ref{section_basic_construction}. This
gives robust coexistense of (volume) non-uniform hyperbolicity and measures with zero Lyapunov exponents.

Bonatti, Gorodetski and \Diaz~have established presence of zero exponents for $C^1$ 
generic diffeomorphisms~\cite{GD, BGD}.
In particular, they have the following.
\begin{theorem}
 For a $C^1$ residual set of diffeomorphisms $f$ from $\diff^1(M)$ every homoclinic class with
a  one dimensional center direction and saddles of different indices is a support of an ergodic  
measure $\mu_f$ with zero center Lyapunov exponent.
\end{theorem}

This result can be viewed as a step towards establishing a generic dichotomy ``uniform hyperbolicity
versus presence of zero Lyapunov exponents.''
Returning to our setup: there is an open and dense subset of $\UL$ for which the homoclinic class is
$\TT^3$. Hence we have the following.
\begin{corollary}
 There exists a $C^1$ residual subset $\mathcal S$ of $\UL$ such that every $f\in\mathcal S$ has an ergodic
measure $\mu_f$ of full support with zero center Lyapunov exponent.
\end{corollary}

Recently Bochi, Bonatti and Diaz~\cite{BBD13} constructed $C^2$ open sets of step skew products (see Definition~\ref{def_ssp}) on any manifold $M$, admitting fully-supported ergodic measures whose Lyapunov exponents along $M$ are all zero. These measures are also approximated by measures supported on periodic orbits. This result is analogous to the result of~\cite{IGKN} and is a major step towards an analogue of Theorem~\ref{th_IGKN} with higher dimensional center foliation.

Finally, in the opposite direction, let us mention the result from~\cite{ABC}: for a $C^1$ generic 
diffeomorphism generic measures supported on isolated homoclinic classes are ergodic and hyperbolic.

\subsection{The setup and the reduction to H\"older skew products}
We will be working in a more general setup than that of Theorem~\ref{th_IGKN}. This is the setup of
partially hyperbolic locally maximal invariant sets rather than partially hyperbolic diffeomorphisms.

Let $h\colon M\to M$ be a diffeomorphism that has locally maximal hyperbolic set $\Lambda$. We
call a skew product $F\colon\Lambda\times S^1\to\Lambda\times S^1$
\begin{equation}
\label{skew_product}
 F(w,x)=(h(w),g_w(x))
\end{equation}
a {\it H\"older skew product} if there exists $C>0$ and $\alpha>0$ such that the fiber diffeomorphisms
satisfy the following inequality
\begin{equation}
 \label{holder_skew_product}
d_{C^0}(g_u,g_v)\le C d(u,v)^\alpha
\end{equation}
for all $u, v\in\Lambda$.

The crucial step in Gorodetski-Ilyashenko strategy is the observation that every $C^1$-small perturbation $f$ 
of $f_0=h\times Id$ in $\diff(M\times S^1)$ is conjugate to a H\"older skew product on the locally
maximal partially hyperbolic set which is homeomorphic to $\Lambda\times S^1$. 

To see this recall that by Hirsch-Pugh-Shub structural stability theorem there exists a locally maximal $W_f^c$-saturated
invariant set $\Delta\subset M\times S^1$ and a homeomorphism $H\colon\Lambda\times S^1\to\Delta$ that ``straightens''
center leaves (that is, $H_*W_{f_0}^c=W_f^c$) and conjugates the induced maps on the space of center leaves (that
is, $H\circ f_0(W_{f_0}^c(\cdot))=f\circ H(W_{f_0}^c(\cdot))$ ). This uniquely defines $H$ as a map on the space of center leaves.
Along the center leaves $H$ can be arbitrary. Therefore we can request that $H$ preserves the second coordinate. In other words,
we choose $H$ so that it has the form
$$
H(w,x)=(H_1(w,x),x).
$$
Define 
$$F_f=H^{-1}\circ f|_{\Delta}\circ H.$$
 Then $F_f$ is a skew product of the form~\eqref{skew_product} which we call
the {\it rectification} of $f$.
\begin{proposition}[\cite{Gor}]
\label{prop_Gor}
 Rectification $F_f$ is a H\"older skew product. Moreover, constants $C$ and $\alpha$ in~\eqref{holder_skew_product} can be chosen 
independently of $f\in\UL$. If $f$ is $C^{1+\varepsilon}$ for some positive $\varepsilon$ then~\eqref{holder_skew_product}
can be replaced with a stronger inequality
$$
d_{C^1}(g_u,g_v)\le C d(u,v)^\alpha
$$
\end{proposition}
Also it is clear that our choice of $H$ ensures that the circle diffeomorphisms $g_w$, $w\in\Lambda$, are $C^1$ close
to identity.
\begin{proof}[Proof (Sketch)]
Consider two nearby points $(u,x), (v,x)\in\Lambda\times S^1$. Then we have
\begin{equation}
\label{est1}
|g_u(x)-g_v(x)|=|f_2(H(u,x))-f_2(H(v,x))|\le A_1d(H(u,x),H(v,x)),
\end{equation}
where $f_2(\cdot)$ stands for the $S^1$-coordinate of $f(\cdot)$ and $A_1$ is a constant 
close to 1 that depends only on $d_{C^1}(f,f_0)$.

Let $D^cf=\|Df|_{E_f^c}\|$. Under appropriate choice of Riemannian metric $|g_w'(x)|=D^cf(H(w,x))$ and 
\begin{multline}
\label{est2}
 |g_u'(x)-g_v'(x)|=|D^cf(H(u,x))-D^cf(H(v,x))|\\ \le A_2 \dist(E_f^c(H(u,x)),E_f^c(H(v,x)))\le
A_2A_3 d(H(u,x),H(v,x))^\beta,
\end{multline}
where $A_2$ is again a constant that depends on $d_{C^1}(f,f_0)$ and the last inequality is due to
H\"older continuity of $E_f^c$.

Combining~\eqref{est1} and~\eqref{est2} we have
$$
d_{C^1}(g_u,g_v)\le A_4\sup_{x\in S^1}d(H(u,x),H(v,x))^\beta
$$
and to obtain the desired inequality~\eqref{holder_skew_product} we need to have
$$
d(H(u,x),H(v,x))\le A_5d(u,v)^\gamma
$$ 
for all $x\in S^1$. The proof of this inequality is not hard and can be carried out in the same way
as the proof of H\"older continuity of the conjugacy from structural stability (see, \eg~\cite[Theorem 19.1.2]{KH}).

Finally let us remark that all the constants above depend only on $d_{C^1}(f,f_0)$ and hence can be chosen uniformly in
$f\in \UL$.
\end{proof}

\subsection{Step skew products} We have passed from a diffeomorphism $f$ to its rectification $F_f\colon\Lambda\times S^1\to\Lambda\times S^1$.
Skew product $F_f$ is easier to handle since dynamics on $\Lambda$ admits symbolic description. Symbolic dynamics takes the 
simplest form in the case when $\Lambda$ is the hyperbolic invariant set of the horseshoe.

Therefore we now assume that $\Lambda$ is the hyperbolic set of the horseshoe. (Ultimately we are interested
in the case $\Lambda=\TT^2$ and we discuss this case later.) Then $\Lambda$ is homeomorphic to $\Sigma^N$,
the space of all bi-infinite words $\omega=\ldots\omega_{-1}\omega_0\omega_1\ldots$,
$\omega_i\in\{1,2, \ldots N\}$, and $h\colon\Lambda\to \Lambda$ becomes the left shift $\sigma\colon\Sigma^N\to\Sigma^N$.
The rectification $F_f$ takes the form 
$$F_f(\omega,x)=(\sigma\omega, g_\omega x).$$

Recall that our goal is to find certain periodic points of $F_f$ that would give us sought measure with zero center exponent in the limit.
Thus we need to come up with periodic words $\omega$, $\sigma^k\omega=\omega$, such that corresponding fiber
diffeomorphisms $g_{\sigma^{p-1}\omega}\circ g_{\sigma^{p-2}\omega}\circ\ldots\circ g_\omega$ can be shown to
posses desired properties such as existence of a fixed point. It would be nice if we can choose diffeomorphism
$g_\omega, g_{\sigma\omega},\ldots , g_{\sigma^{p-1}\omega}$ independently to produce these properties. However
the problem is that, a priori, these diffeomorphisms do depend on each other. This heuristics motivates the 
introduction of step skew products.
\begin{definition}\label{def_ssp}
 A skew product $F\colon\Sigma^N\times S^1\to\Sigma^N\times S^1$ is called step skew product if it has the 
form
$$
F(\omega,x)=(\sigma\omega, g_{\omega_0} x),
$$
where $\omega_0$ is the zeroth letter of $\omega$.
\end{definition}
\begin{remark}
 The term ``step skew product'' comes from the analogy with a step function as $\omega\mapsto g_{\omega_0}$ 
takes only finitely many values. Another common term is ``iterated function system."
\end{remark}
For a step skew product the fiber diffeomorphism $g_{\sigma^{m-1}\omega}\circ g_{\sigma^{m-2}\omega}\circ\ldots\circ g_\omega$
becomes $g_{\omega_{m-1}}\circ g_{\omega_{m-2}}\circ\ldots\circ g_{\omega_0}$ and thus we can paste together
diffeomorphisms $g_1, g_2,\ldots g_N$ any way we please to obtain orbits of $F$ with desired properties.
\begin{example}
\label{example}
 Take $N=2$ and let $g_1$ be a rotation by a very small angle and $g_2$ be a diffeomorphism with two
fixed points, an attractor and a repeller. This choice makes it easy to create periodic words of $g_1$ and
$g_2$ that give periodic points with small (positive or negative) center exponents. Indeed, take a small 
interval $I\subset S^1$ and start rotating it using $g_1$. Once the interval is near the attracting fixed point of 
$g_2$ we can slightly shrink it with $g_2$ and then continue rotating  until it comes inside the original interval
$I$. Then there must be a fixed point in $I$ for corresponding composition of $g_1$-s and $g_2$-s. This point is
a periodic point with small center exponent for the step skew product $F$. 
\end{example}
This type of arguments are very fruitful and can go a long way. In fact, for the above example, one can show 
that periodic points with arbitrarily small center exponent are dense in $\Sigma^2\times S^1$. Also
one can construct dense orbits of $F$ with prescribed center exponent $\lambda$ in some small interval
$(-\varepsilon,\varepsilon)$. These properties are $C^1$ stable in the space of step skew products and,
more importantly, in the space of H\"older skew products~\cite{IG1, IG2}.

\subsection{Sketch of the proof of Theorem~\ref{th_IGKN}.} The first step is to choose certain step skew product $F$ by 
specifying
the circle diffeomorphisms $g_1, g_2, \ldots g_N$. Then there exists a diffeomorphism
$f\colon M\times S^1\to M\times S^1$ such that $F_f=F$. Diffeomorphisms $g_1, g_2, \ldots g_N$ are chosen so that we can find a sequence of periodic points
$\{p_n;n\ge1\}$ which gives an ergodic measure $\mu$, $\supp\mu=\Sigma^N\times S^1$, with zero center
exponent in the limit (cf. Example~\ref{example}).

One has to prove that the construction of $\{p_n;n\ge1\}$ is robust under sufficiently 
$C^1$-small perturbations of $F$ in the space of H\"older skew products. So that, by Proposition~\ref{prop_Gor}, if $\mathcal V$ is 
sufficiently small
$C^1$ neighborhood of $f$ then any $g\in\mathcal V$ would have an ergodic measure with zero
center exponent supported on locally maximal partially hyperbolic set.

Next give an outline of this argument without going into technical details.

For any periodic point $p$ of $F$ let $\lambda^c(p)$ be the center exponent at $p$ and let $\nu(p)$ be the atomic
measure supported on the orbit of $p$, that is,  
$$
\nu(p)=\frac{1}{|\mathcal O(p)|}\sum_{q\in\mathcal O(p)}\delta_q.
$$

We have to construct a sequence of periodic points $\{p_n;n\ge 1\}$ such that $\lambda^c(p_n)\to0$, $n\to\infty$,
and any weak$^*$ accumulation point $\mu$ of the sequence $\{\nu(p_n); n\ge 1\}$ is an ergodic
measure of full support. Note that, since the space of measures on a compact space is compact in weak$^*$
topology, the sequence $\{\nu(p_n); n\ge 1\}$ has at least one accumulation point $\mu$. Measure $\mu$ 
is the measure we seek. Indeed, by the Birkhoff ergodic theorem
$$
\lambda^c(\mu)=\int_{\Sigma^N\times S^1}\log\frac{\partial F}{\partial x}d\mu=\lim_{n\to\infty}\int_{\Sigma^N\times S^1}\log\frac{\partial F}{\partial x}d\nu(p_n)=\lim_{n\to\infty}\lambda^c(p_n)=0.
$$

Periodic points $p_n$ are constructed inductively. Ergodicity of the limit is guaranteed by certain
similarity condition on periodic orbits $\mathcal O(p_n)$, $n\ge1$. Very roughly, this condition
says that for large $n$ and all $m>n$ the (non-invariant) measures
$$
\frac1n\sum_{i=1}^{n}\delta_{F^i(q)}, \;\;\;\;q\in\mathcal O(p_m),
$$
are weak$^*$ close to $\nu(p_n)$ for majority of points $q\in\mathcal O(p_m)$.

The $\Sigma^N$-coordinate of a periodic point $p_n$ is a periodic word with some period $\alpha$. Then
$\Sigma^N$-coordinate of the next periodic point $p_{n+1}$ is a periodic word with period $\alpha^k\beta$, where $\beta$ is 
much shorter than $\alpha^k$. Word $\beta$ is a ``correction term'' that, in particular, ensures inequality
$|\lambda^c(p_{n+1})|<c|\lambda^c(p_n)|$ for some fixed $c\in(0,1)$.

Another thing to take care of in the inductive construction is to make sure that orbits $\mathcal O(p_n)$ 
become
well-distributed in $\Lambda\times S^1$ to guarantee $\supp\mu=\Lambda\times S^1$. For words this means 
that every finite word $\gamma$ eventually appears as a subword of $\alpha=\alpha(p_n)$. Note that once we pass to $\alpha^k\beta$
word $\gamma$ would appear $k$ times. Thus $\gamma$ will maintain the same positive proportion in all subsequent words
and corresponding open subset of $\Lambda$ (the $\gamma$-cylinder) will have some positive $\mu$-measure. (Of course,
one has to take care of $S^1$-coordinate as well.)

We see that there are many things that has to be carefully tracked during the induction step. Moreover, the 
procedure must be $C^1$ robust. For step skew products such inductive procedure was carried out in~\cite{IGKN} and it is 
sufficient to use only two diffeomorphisms and work over $\Sigma^2$.
It is harder to carry out this scheme $C^1$-robustly in the space of H\"older skew products, this was done in~\cite{KN} who need at least 5 symbols to play with.

So we consider a $C^1$-small perturbation $\tilde F$ , $\tilde F(\omega,x)=(\sigma\omega,\tilde g_\omega x)$,
of the step skew product $F$.  The difficulty is that circle diffeomorphisms $\tilde g_\omega$ depend 
on the whole word $\omega$.
This difficulty can be overcome using so called ``predictability property'' of H\"older skew products: 
 for any $m\ge 1$ the composition 
$\tilde g_{\sigma^{m-1}\omega}\circ \tilde g_{\sigma^{m-2}\omega}\circ\ldots\circ\tilde g_{\omega}$ can be
determined approximately from first $m$ letters of $\omega$. More precisely, for any $m\ge 1$, given two words
$\omega$ and $\omega'$ with $\omega_i=\omega_i'$ for $i=0, 1, \ldots m-1$ we have
$$
d_{C^0}(\tilde g_{\sigma^{m-1}\omega}\circ \tilde g_{\sigma^{m-2}\omega}\circ\ldots\circ\tilde g_{\omega},
\tilde g_{\sigma^{m-1}\omega'}\circ \tilde g_{\sigma^{m-2}\omega'}\circ\ldots\circ\tilde g_{\omega'})\le K\delta^\beta,
$$
where $\delta=d_{C^1}(F,\tilde F)$; $K>0$ and $\beta>0$ are some fixed constants and $\sigma$ is the left shift. 

\subsection{The case $\Lambda=\TT^2$} In the case when $\Lambda=\TT^2$ and $h\colon\TT^2\to\TT^2$
is an Anosov automorphism $A\colon\TT^2\to\TT^2$ Gorodetski-Ilyashenko strategy cannot be applied in a straightforward way.
One still has symbolic dynamics which is now a subshift of finite type. The major problem is that step skew products
cannot be realized as partially hyperbolic diffeomorphisms contrary to the horseshoe case. 

One way
around this difficulty is to pass to a power $A^m$ of $A$ so that there exists an embedded $A^m$-invariant
horseshoe $\bar\Lambda\subset\TT^2$. Then define a skew product $F$ over $A^m$ as a step skew product over
$\bar\Lambda$ and extend smoothly to the rest of $\TT^2$. Then the strategy outlined above goes through
and the only alteration to be made is to make sure that periodic orbits $\mathcal O(p_n)$ spend some time
in the complement of $\bar\Lambda\times S^1$ to guarantee that $\supp\mu=\TT^3$. Such arguments
appear in~\cite{nalsky} where the author discusses the case when $\Lambda$ is the Smale-Williams solenoid.

%===============================================================================

\section{Non-zero Lyapunov exponents and pathological foliations}
\label{section_pathol}

Presence of non-zero Lyapunov exponents may lead to certain measure-theoretic ``pathology" of the center foliation of a partially hyperbolic diffeomorphism. This observation is due to \Mane$~$ and first appeared in~\cite{SW}. 

Let $f\colon\TT^3\to\TT^3$ be the partially hyperbolic diffeomorphism constructed in Section~\ref{section_basic_construction}. By Oseledets Theorem
there exists a full volume set $\Lambda\subset\TT^3$ of Lyapunov regular points whose center exponent $\lambda^c=\lambda^c(f)$ is positive. Provided that $f$ is sufficiently close to the linear automorphism $L$ the center distribution $E^c$ integrates to a foliation $W^c$ which is also a circle fibration.

Recall that we can decompose $\Lambda$ as $\Lambda=\cup_{k\ge 1}\Lambda_k$ so that on sets $\Lambda_k$ (called {\it Pesin sets}) we have uniform hyperbolicity. That is, $\forall x\in\Lambda_k$ and $\forall n>0$
$$
\|Df^n v\|\ge \frac1k e^{n(\lambda^c-\varepsilon)}\|v\|,\;\;\; v\in E^c,
$$
where $\varepsilon\in(0,\lambda^c)$.

Each leaf $\mathcal C\in W^c$ intersects a set $\Lambda_k$, $k\ge 1$, at a set of leaf Lebesgue measure zero since otherwise the lengths of 
$f^n(\mathcal C)$ would grow to infinity. {\it Thus (a full volume set) $\Lambda$ intersects every leaf of $W^c$ at a set of leaf measure zero.}

This argument can be generalized to the higher dimensional setup and gives the following result.

\begin{theorem}[\cite{PH}]\label{th_HirPes} Let $f$ be a $C^2$ volume preserving partially hyperbolic diffeomorphism. Assume that the center distribution $E^c$ integrates to an invariant foliation $W^c$ with compact leaves. Also assume that $f$ is $W^c$-dissipative, that is, the sum of center exponents is different from zero on a set of full volume. Then there exists a set of positive volume that meets every leaf of $W^c$ at a set of zero leaf volume.
\end{theorem}

\begin{remark}\label{rm_katok}
 The first example of pathological center foliation was constructed by A. Katok in early nineties. This examples also lives on $\TT^3$, but has zero center exponents. See, \eg~\cite{PH, G} for a description of Katok's example and~\cite{M} for a different (non-dynamical) construction of a pathological foliation that was inspired by Katok's construction.
\end{remark}

\subsection{Conditional measures and absolute continuity}
There are several way to define absolute continuity of foliations. We say that
 a foliation is  absolutely continuous if for each foliation box the factor measure of the volume is absolutely continuous with respect to the induced volume on a transversal and the conditional measures of the volume on the plaques of the foliation are absolutely continuous with respect to the Lebesgue measure on the plaques. This property is weaker than absolute continuity of holonomy maps. However, it can be proved that  the above absolute continuity property implies that the holonomy maps between almost all pairs of transversals are absolutely continuous.
 Note that the definition of absolute continuity is independent of the particular choice of the volume form. Also note that the definition does not require presence of any dynamics. See~\cite{Pes2} for a detailed discussion of absolute continuity.

\subsection{Pathological foliations with compact center leaves}

If $M$ is the total space of a fiber bundle with compact fibers then one can speak about conditional measures on the fibers without considering foliation boxes. This is because in this case the parition into fibers is measurable.  Repetition of the argument given in the beginning of the current section immediately yields the following result.

\begin{theorem} \label{th_non_abs_cont} Consider a partially hyperbolic diffeomorphism $f$ whose center foliation is a circle fibration. Assume that
$f$ preserves an ergodic measure $\mu$ (\eg  volume) with negative (or positive) center Lyapunov exponent. Then the conditional measures of $\mu$ on the leaves of the center foliation are singular with respect to the Lebesgue measure on the leaves.
\end{theorem}

This result was generalized in~\cite{PH} to the setting of partially hyperbolic diffeomorphisms with higher dimensional compact center leaves. If the sum of center Lyapunov exponents is non-zero then the center foliation is non-absolutely continuous. The sum of center Lyapunov exponents can be perturbed away from zero by Theorem~\ref{th_BB}.

Next theorem generalizes Theorem~\ref{th_non_abs_cont} and describes the conditional measures on the leaves of the center foliation.

\begin{theorem}[\cite{ruelle-wilkinson, PH}] \label{th_ruelle_wilkinson} Consider a dynamically 
coherent  partially hyperbolic diffeomorphism $f$ whose 
center leaves are fibers of a (continuous) fiber bundle. Assume that the all center Lyapunov exponents are negative (or positive) then
the conditional measures of $\mu$ on the leaves of the center foliation are atomic with $p$, $p\ge 1$, 
atoms of equal weight on each leaf.
\end{theorem} 

It is interesting to obtain examples beyond circle fibrations to which Theorem~\ref{th_ruelle_wilkinson} applies.
\begin{question}
%\marginpar{veirying Avila-Viana wilkinson recent manuscript. Ali, so is it available in some form?}
Start with a partially hyperbolic skew product with fibers of dimension $\ge 2$ that has zero center Lyapunov exponents. Is there an ergodic perturbation with only negative center  Lyapunov exponents? 
%\marginpar{The answer is yes.  Firstly make a perturbation to make all center exponents negative 
%in average and also stable accessibility. By Burns-Dolgopyat-Pesin, you have ergodicity. Ali, I don't understand it, how does that guarantee that all expoenents are negative?}
\end{question}

%\textit{
%Related to the above question,  consider the partially hyperbolic diffeomorphism  $f_0= A \times Id : \mathbb{T}^4 \rightarrow \mathbb{T}^4$ where $A$ is Anosov and the center bundle is two dimensional. Then after small $C^1$ perturbation we find $f$ such that $\int_{\mathbb{T}^4} log J^cf dm < 0.$ Perturbing again we can assume that $f$ is stably ergodic and stably the same dominated splitting. After a small perturbation we make the two center exponents negative.}
%

Theorem~\ref{th_ruelle_wilkinson} is a ``fiber bundle''-generalization of the following proposition. 

\begin{proposition} \label{prop_atomic_measure} Let $f$ be a diffeomorphism of a compact manifold that preserves an ergodic measure $\mu$. Assume that all Lyapunov exponents of $\mu$ are negative (or positive). Then $\mu$ is an atomic measure on a periodic sink (or source).
\end{proposition} 
%\marginpar{comment on katoks paper, which paper?}
To illustrate the basic idea we prove the proposition first and then Theorem~\ref{th_ruelle_wilkinson}.

\begin{proof}[Proof of Proposition~\ref{prop_atomic_measure}] Consider a Pesin set $\Lambda_k$ of positive measure. Since all Lyapunov exponents are negative the stable manifolds of points in $\Lambda_k$ are open balls in $M$. Let $\delta_k>0$ be the lower bound on the size of the stable manifolds of $x\in \Lambda_k$.

Choose a ball $B$ of radius $\delta_k/10$ such that $\mu(B\cap\Lambda_k)>0$. Then, by Poincare recurrence, there exists an $x\in B\cap\Lambda_k$ that returns to $B\cap\Lambda_k$ infinitely often. The ball $B(x,\delta_k)$ is in the stable manifold of $x$. Thus, for sufficiently large return time $n_m$, 
$$f^{n_m}(B(x,\delta_k))\subset B(f^{n_m}(x),\delta_k/2)\subset B(x,\delta_k).$$ 

Since $f^{n_m}$ preserves $\mu$ we conclude that the restriction of $\mu$ to $B(x,\delta_k)$ is an atom at $y=\cap_{i\ge 0} f^{in_m}(B(x,\delta_k))$. Then $y$ must be periodic and ergodicity implies that $\mu$ is an atomic measure that sits on the orbit of $y$.
\end{proof}

\begin{proof}[Proof of Theorem~\ref{th_ruelle_wilkinson}]
As earlier we consider Pesin sets $\Lambda_k$, $k\ge 1$, and we denote by $\delta_k>0$ the lower bound on the size of stable manifolds of $x\in
\Lambda_k$.  To avoid possible confusion we point out that in this proof we use term ``stable manifold" for Pesin stable manifold as opposed to the manifold tangent to $E^s$.

By $B^c(x,r)$ we denote a ball inside $W^c(x)$ centered at $x$ of radius $r$. Since the dimension of stable manifolds is equal to $\dim E^c+\dim E^s$, a simple growth argument (which utilizes dynamical coherence) shows that the stable manifolds are contained in center-stable leaves. Thus, for $x\in\Lambda_k$ and sufficiently small $r$, \eg $r=\delta_k/2$ the ball $B^c(x, r)$ is contained in the stable manifold of $x$.%\footnote{Ali, I think d.c. is important for this. Pesin-Hiroyama completely take this fact for granted and don't pose d.c. assumption which I think is a minor mistake. In fact, they don't even assume that center foliation is f-invariant.}

By $\mu_x, x\in M$, we denote the conditional measure on the center leaf $W^c(x)$. By invariance, we have
$$
\mu_{f(x)}=f_*\mu_x
$$
for $\mu$ a.e. $x\in M$. 

Recall that the union of Pesin sets $\cup_{k\ge 1} \Lambda_k$ has full measure. Choose a sufficiently large $k$ so that $\mu(\Lambda_k)>1/2$. Then the set
$$
A=\bigcup_{x\in M:\mu_x(W^c(x)\cap\Lambda_k)\ge 1/2}W^c(x)
$$
has positive measure. Consider the set $B\subset A$ that consists of points that return to $A$ infinitely many times both in positive and negative time. Clearly $B$ is saturated by the center leaves and, by the Poincar\'e recurrence, $\mu(B)=\mu(A)$. Let $F\colon B\to B$ be the first return map of $f$. Then, obviously, $F$ is a bijection, $F(W^c(x))=W^c(F(x))$ and $F_*\mu_x=\mu_{F(x)}$.

There exists a number $m$ such that every center leaf can be covered by at most $m$ balls of radius $\delta_k/2$. Thus for every $x\in B$ a we can pick a ball $B^c_x=B^c(y(x),\delta_k/2)$ such that $y(x)\in W^c(x)$, $B^c_x\cap\Lambda_k\neq\varnothing$ and $\mu_x(B^c_x)\ge1/2m$.

Now let $B^c_{n,x}=F^n(B^c_{F^{-n}(x)})$ and note that $\mu_x(B^c_{n,x})=\mu_{F^{-n}(x)}(B^c_{F^{-n}(x)})\ge1/2m$. At the same time $\diam(B^c_{n,x})\to0$ as $n\to+\infty$ uniformly in $x$ since every ball $B^c_x$ is contained in a stable manifold of some point from $\Lambda_k$.

We conclude that for every $x\in B$ we have a sequence of shrinking sets inside $W^c(x)$ of positive $\mu_x$ measure. Clearly, a converging subsequence gives an atom of $\mu_x$. Thus for $\mu$ a.e $x\in B$ measure $\mu_x$ has an atom of weight $\ge 1/2m$.

Measure $\mu$ can be decomposed as a sum of measure $\mu^a$ which has atomic conditional measures on the center leaves and measure $\mu^{na}$ which has atomless conditional measures on the center leaves. Both measures are invariant and we have shown that $\mu^a(M)>0$. Thus ergodicity of $\mu$ implies that $\mu^{na}$ is a trivial measure, \ie $\mu^{na}(M)=0$. Furthermore, in similar fashion it follows from ergodicity that for $\mu$ a.e. $x$ measure $\mu_x$ has $p$ atoms of equal weight, where $p$ does not depend on $x$.
\end{proof}

\begin{remark}
 It is easy to construct examples with higher dimensional center leaves and non-zero average center exponent for which the conditional measures are singular but non-atomic.
\end{remark}
\begin{remark}
   Note that the proof of Theorem~\ref{th_ruelle_wilkinson} does not provide information on the 
number of atoms $p$. It was shown in~\cite{Homburg} that $p$ can be any positive integer. More precisely, 
it was proved that if an ergodic volume preserving perturbation $L$ of partially hyperbolic automorphism 
$L_0\colon\TT^3\to\TT^3$ ($L_0$ is given by~\eqref{L0}) 
has 
\begin{enumerate}
\item negative center exponents
\item a fixed center fiber with a unique hyperbolic attracting fixed point and a unique hyperbolic repelling fixed point
\item minimal strong unstable and strong stable foliations
\end{enumerate}
    then $p=1$. The assumptions on $L$ are known to hold for certain perturbations of the form~\eqref{Lphi}.

Now, by passing to finite self covers of $\TT^3$ one can get partially hyperbolic diffeomorphisms with 
prescribed number of atoms $p$.
\end{remark}

\subsection{Pathological foliations with non-compact leaves: near geodesic flow}

For a general partially hyperbolic diffeomorphism the geometric structure of the support of disintegration measures is not clear.

Avila, Viana and Wilkinson~\cite{AVW} showed that, for volume preserving perturbations of a time-one map $f_0\colon N\to N$ of the geodesic flow  a closed negatively curved surfaces,  the center foliation $W^c$ has either atomic or absolutely continuous conditional measures. Recall that, as discussed in Section~\ref{section_srb}, center Lyapunov exponent of $f_0$ can be perturbed away from zero. 
%It was shown in~\cite{AVW} that the dichotomy
%\begin{center} ``zero vs. non-zero center Lyapunov exponent''
% \end{center}
%is, actually, equivalent to the dichotomy
%\begin{center} ``atomic vs. absolutely continuous center conditional measures''
%\end{center}
%It was shown in~\cite{AVW} that there is a dichotomy between atomic and Lebesgue disintegration. 
The center Lyapunov exponent plays a key role in the proof of the ``atomic-Lebesgue dichotomy'' of~\cite{AVW}.

 The fact that non-zero center exponent leads to atomic conditional measure on the center leaves follows easily from (the more general form of) Theorem~\ref{th_ruelle_wilkinson}. Indeed, for a perturbation $f\colon N\to N$ of $f_0$ one can consider the extension
$$
\xymatrix{
S^1\ar[d] & S^1\ar[d]\\
 E\ar[r]^{\hat f}\ar[d] &E\ar[d]\\
 N\ar[r]^f& N
}
$$
where fibers $S^1_x$, $x\in N$, are the segments of center leaves with end-points identified, \ie $S^1_x=[x,f(x)]^c/x\sim f(x)$, and $\hat f$ is induced by $f$ the obvious way. This way non-compact center foliation can be compactified and~\cite{ruelle-wilkinson} applies to give atomic conditionals on the fibers.

In the zero exponent case, the authors use in a subtle way the invariance principle proved in~\cite{ASV} (see also~\cite{AV}). The invariance principle is a non-trivial extension of seminal work of Ledrappier~\cite{led} to the context of smooth cocycles. Assuming that the center exponent is zero, the invariance principle yields a dichotomy into two cases; these two cases are characterized in terms of  geometry of the supports of conditional measures along center leaves. In the first case, the support of conditional measures is a countable subset and there exist a full volume subset such that intersects almost all center leaves in exactly $k$ orbits.
In the second case, the supports of  conditional measures coincide with the center leaves. In this case, using invariance of disintegration with respect to holonomies and regularity of holonomies (stable and unstable holonomies are $C^1$ between center leaves) the authors are able to prove that the conditional measures are absolutely continuous. Moreover, they prove that in this case, the diffeomorphism is time one map of a smooth flow and consequently, center foliation is smooth.

\begin{question} Is there an example $f$ close to time one map of geodesic flow of a negatively curved surface, such that center exponents of $f$ vanish and the center foliation is singular?
\end{question}
%
% We emphasize that the key property for the diffeomorphisms near to time-one map of Anosov flow is that they are partially hyperbolic and moreover all center leaves are fixed by the dynamics. This implies that in the atomic case, the disintegrated measures do have infinitely many atoms. Indeed, if $f$ is such a partially hyperbolic diffeomorphism $f(W^c(x)) = W^c(x)$ and  for any $a$ atom, $f^n(a), n \in \mathbb{Z}$ are atoms of the disintegration of volume along $W^c.$

\subsection{Pathological foliations with non-compact leaves: Anosov and derived from Anosov case.}
It turns out that the argument given in the beginning of current section can be generalized to certain
non-compact center foliations, which we proceed to describe.

Let $L$ be an automorphism of $\TT^3$ with three distinct Lyapunov exponents $\lambda^s(L)<\lambda^c(L)
<\lambda^u(L)$. Assume that $\lambda^c(L)>0$. Then the center foliation $W_L^c$ is an expanding foliation.
We have shown in Section~\ref{subs_BB} that there exists a perturbation $g$ such that $\lambda^c(g)>\lambda^c(L)$.
Define {\it geometric growth} of $W_g^c$ as
$$
\chi^c(g)=\lim_{n\to\infty}\frac1n\log \mbox{length} (g^n(B_g^c(x,r))),
$$
where $B_g^c(x,r)$ is a ball inside $W_g^c(x)$. Automorphism $L$ is Anosov and thus, by the structural 
stability, $L$ and $g$ are conjugate. One can easily deduce that $\chi^c(g)=\chi^c(L)=\lambda^c(L)$.
This property of $g$ plays the same role as compactness of center leaves earlier. We can essentially
repeat the argument given in the beginning of current section and conclude that the full volume set of
Lyapunov regular points of $g$ intersects every leaf of $W_g^c$ by a set of zero length.

These arguments can be generalized to higher dimensional automorphisms. Namely, assume that automorphism $L$
preserves a partially hyperbolic splitting $T\TT^d=E^s_L\oplus E^c_L\oplus E^u_L$ and that the center
foliation $W^c_L$ is uniformly expanding.
\begin{theorem}[\cite{SX}]
 If $g\in \diff_m(\TT^d)$ is a diffeomorphism $C^1$-close to $L$ and
$$
\lambda^c(g)=\int_{\TT^d}\log J^c(g)dm>\int_{\TT^d}\log J^c(L)dm=\lambda^c(L),
$$
then foliation $W_g^c$ is non-absolutely continuous.
\end{theorem}

Note that one can also consider diffeomorphisms $g$ with $\lambda^c(g)<\lambda^c(L)$ and it is not clear if
this inequality forces non-absolute continuity of $W_g^c$. However it is the case when $d=3$. In the 
three dimensional case there is an if-and-only-if description for non-absolute continuity of $W_g^c$ 
given in terms of eigenvalues at periodic points of $g$~\cite{G12}. The following theorem is a corollary of this description.
\begin{theorem}[\cite{G12}]
Let $L$ be an automorphism of $\TT^3$ with three distinct Lyapunov exponents $\lambda^s(L)<\lambda^c(L)
<\lambda^u(L)$. Assume that $\lambda^c(L)>0$. Let $\UL$ be a $C^1$-small neighborhood of $L$ in $\diff_m(\TT^3)$.
Then there is a $C^1$-open and $C^r$-dense set $\mathcal V\subset \UL$ such that $g\in \mathcal V$
if and only if the center foliation $W_g^c$ is non-absolutely continuous with respect to the volume $m$.
\end{theorem}
 Recently, R. Var\~ao~\cite{Va} showed that there exist Anosov diffeomorphisms in $\mathcal V$ with non-absolutely continuous center foliation which does not have atomic disintegration.

 %Any $A \in SL(3, \mathbb{Z})$ with at least one eigenvalue with norm larger than one, induces a linear partially hyperbolic diffeomorphism on $\mathbb{T}^3.$ Recall that for any partially hyperbolic diffeomorphism $f$ on $\mathbb{T}^3,$ there exist a unique linear diffeomorphism $A,$ such that $A$ induces the same automorphism as $f$ on the fundamental group $\pi_1(\mathbb{T}^3).$    

%Let $f: \mathbb{T}^3 \rightarrow \mathbb{T}^{3}$ be a partially hyperbolic diffeomorphism. Consider $f_* : \mathbb{Z}^3 \rightarrow \mathbb{Z}^3$ the action of $f$ on the fundamental group of $\mathbb{T}^3.$ $f_*$ can be extended to $\mathbb{R}^3$ and the extension is the lift of  a unique linear automorphism $A :  \mathbb{T}^{3} \rightarrow \mathbb{T}^{3}$ which is called the linearization of $f.$  

So far we have concentrated on diffeomorphism which are $C^1$ close to $L$.
Recall that a diffeomorphism $f$ is {\it derived from Anosov} (or simply DA) diffeomorphism if it is homotopic to an Anosov diffeomorphism. The following recent result of Ponce, Tahzibi and Var\~ao~\cite{PTV} establishes atomic disinegration for a class of (non-Anosov) partially hyperbolic DA diffeomorphisms.

\begin{theorem} 
Let $L$ be a hyperbolic automorphism of $\TT^3$ with three distinct Lyapunov exponents $\lambda^s(L)<\lambda^c(L)
<\lambda^u(L)$. 
Let $f\colon\mathbb{T}^3 \rightarrow \mathbb{T}^3$ be volume preserving DA diffeomorphism (homotopic to $L$). Assume that $f$ is partially hyperbolic, volume preserving and ergodic.  Also assume that $\lambda^c(f) \lambda^c(L) < 0$ then the disintegration of volume along center leaves of $f$ is atomic and in fact there is just one atom per leaf. 
\end{theorem} 

The above theorem  can be verified for an open class of diffeomorphisms found by  Ponce-Tahzibi in \cite{PT}.

It is interesting to emphasize that conservative ergodic DA diffeomorphism on $\mathbb T^3$ show a feature that is not, so far, shared with any other known partially hyperbolic diffeomorphims on dimension three, it admits all three disintegration of volume on the center leaf, namely: Lebesgue, atomic, and singular non-atomic (by a recent result of R. Var\~ao~\cite{Va}. We also remark that Katok's example (cf. Remark~\ref{rm_katok}) can be modified to give singular non-atomic conditionals. However this example is not ergodic.

\end{document}